%% file: main.tex
\documentclass[reqno,a4paper,12pt]{amsart}

\usepackage{styles/mystyle}
\input{styles/commands.tex}
\input{styles/theoremstyles.tex}

\input{styles/title.tex}

\begin{document}
    \input{chapters/abstract.tex}
    
    \maketitle

    \input{chapters/introduction.tex}

    \input{chapters/dyadicCubes.tex}

    \input{chapters/algorithm.tex}

    \input{chapters/tools.tex}

    \input{chapters/manifolds.tex}

    \input{chapters/questions.tex}

    \bibliographystyle{amsplain}
    \nocite{*}
    \bibliography{bibliography/references}
\end{document}

%% file: styles/commands.tex
\newcommand{\N}{\mathbb{N}}

\newcommand{\R}{\mathbb{R}}
\newcommand{\HR}{\mathbb{H}}

\newcommand{\T}{\mathbb{T}}

\newcommand{\B}{\mathbb{B}}
\newcommand{\Sp}{\mathbb{S}}
\newcommand{\Rd}{\R^d}

\newcommand{\conv}{\operatorname{conv}}
\newcommand{\diam}{\operatorname{diam}}
\newcommand{\len}{\operatorname{len}}
\newcommand{\Int}{\operatorname{Int}}
\newcommand{\dist}{\operatorname{dist}}

\newcommand{\cubedec}{cube-\hspace{0pt}decomposable\xspace}
\newcommand{\Cubedec}{Cube-\hspace{0pt}decomposable\xspace}

%% file: styles/theoremstyles.tex
\theoremstyle{plain}

\newtheorem{theorem}{Theorem}[section]
\newtheorem{proposition}[theorem]{Proposition}
\newtheorem{lemma}[theorem]{Lemma}
\newtheorem{corollary}[theorem]{Corollary}

\newtheorem*{claim*}{Claim}

\newtheoremstyle{definitionstyle}%
    {}{}%
    {\normalfont}%
    {}%
    {\itshape}{.}%
    { }%
    {}

\theoremstyle{definitionstyle}

\newtheorem{definition}[theorem]{Definition}
\newtheorem*{definition*}{Definition}

\theoremstyle{definition}
\newtheorem{definitionbold}[theorem]{Definition}

\theoremstyle{remark}

\newtheorem*{remarks*}{Remarks}
\newtheorem*{remark}{Remark}

\crefname{theorem}{Theorem}{Theorems}
\crefname{proposition}{Proposition}{Propositions}
\crefname{lemma}{Lemma}{Lemmas}
\crefname{corollary}{Corollary}{Corollaries}
\crefname{definition}{Definition}{Definitions}
\crefname{example}{Example}{Examples}

\crefname{definitionbold}{Definition}{Definitions}

\crefname{section}{Section}{Sections}

\AtBeginEnvironment{proposition}{\crefalias{theorem}{proposition}}
\AtBeginEnvironment{lemma}{\crefalias{theorem}{lemma}}
\AtBeginEnvironment{corollary}{\crefalias{theorem}{corollary}}
\AtBeginEnvironment{definition}{\crefalias{theorem}{definition}}
\AtBeginEnvironment{definitionbold}{\crefalias{theorem}{definition}}
\AtBeginEnvironment{example}{\crefalias{theorem}{example}}

%% file: styles/title.tex
\title{Optimal-diameter partitions into regions of prescribed measure}

\author{Grigory Voinov}

%% file: chapters/abstract.tex
\subjclass[2020]{Primary 30L05; Secondary 28A75, 52C22, 57Q15}
	\keywords
    {
        partition,
        metric measure space,
        Ahlfors--David regular,
        dyadic cubes,
        \(C^1\) manifolds.
    }

\begin{abstract}
    In this paper we find a sufficient condition on
    an Ahlfors--David regular metric
    measure space under which it admits a partition into parts of
    prescribed measures and
    optimal (up to a constant) diameters.
    The proof uses the construction of dyadic cubes.
    The process is algorithmic: the pieces are cut
    out one by one via a filling procedure on the tree of
    dyadic cubes.
    We introduce the notion of spaces which admit a
    connected dyadic cube decomposition and prove that they
    admit a partition of the kind described above.
    We then develop several techniques to obtain such
    spaces and show some natural examples of this type.
\end{abstract}

%% file: chapters/introduction.tex
\section{Introduction} \label{sec: introduction}

    The problem of partitioning a space into regions of prescribed measure
    and small diameter was first considered by K. Stolarsky,
    who in \cite{Sto73} observed
    that for every natural number \(N\) the unit sphere \(\Sp^d\)
    admits a partition into \(N\) regions of equal volume and diameter at
    most \(c_d N^{-1/d}\) for some constant \(c_d > 0\) independent of \(N\).
    An explicit recursive construction of such an equal-area partition of
    the sphere with small diameter was subsequently given by P. Leopardi
    \cite{Leo06}.
    G. Gigante and P. Leopardi in \cite{GL17} proved the same result for
    arbitrary connected Ahlfors--David regular metric measure spaces.

    Separately, M. Skriganov \cite{Skr17} established a constructive proof
    of the existence of equal-measure partitions with small average diameter
    for compact \(d\)-rectifiable metric spaces.
    More recently, J. Kitagawa and A. Takatsu \cite{KT25} gave an alternative
    proof of the existence of equal-measure partitions of the sphere with
    diameter bounds, using optimal transport methods.

    \medskip

    The idea of considering spanning trees of Voronoi cells
    in the context of partitioning a sphere
    was suggested by U. Feige and G. Schechtman in [\cite{FS02}, Lemma 21].
    Later, in \cite{GL17}
    G. Gigante and P. Leopardi applied this technique to dyadic cubes.

    \smallskip

    The construction of
    dyadic cubes was developed by G. David and S. Semmes in
    \cite{Dav88, DS97}
    for Ahlfors--David regular subsets of Euclidean spaces.
    The Assouad embedding theorem \cite{Ass83} shows that David's
    approach extends to arbitrary Ahlfors--David regular metric
    measure spaces.
    M. Christ \cite{Chr90} subsequently gave a direct
    construction of dyadic cubes for the more general setting of
    spaces of homogeneous type.
    This construction was later refined for doubling quasimetric spaces
    by T. Hyt\"onen and A. Kairema \cite{HK12}.

    \smallskip

    In this article we partially generalise the result of \cite{GL17} from a
    partition into \(n\) parts of equal measure to a partition into parts
    of arbitrarily prescribed measures \(\lambda_1, \dots, \lambda_n\)
    for sufficiently regular spaces, with diameter control.

    \begin{definitionbold} \label{def: weightedDecomp}
        Suppose that \((X, \rho)\) is a metric space
        with a Borel measure \(\mu\), such that
        \(0 < \mu(X) < \infty\).
        We say that \(X\) admits a 
        \textbf{weighted decomposition}
        if there exist constants \(d, C^* > 0\),
        depending only on \((X, \rho, \mu)\), such
        that the following holds.
        For all positive numbers
        \(\lambda_1, \dots, \lambda_n\) with
        \(\sum_{j=1}^n \lambda_j = \mu(X)\), there exist measurable subsets
        \(X_1, \dots, X_n \subset X\)
        such that \(\bigsqcup_{j=1}^n X_j= X, \, \mu(X_j)= \lambda_j,\)
        and each \(X_j\) is contained in a ball of radius
        \(C^* \lambda_j^{\frac{1}{d}}\).
        Such a constant \(d\) is called an \textbf{order} of
        the decomposition.
    \end{definitionbold}

    The main result of the paper may be formulated
    as follows.

    \begin{theorem} \label{thm: partitionConnectedIntro}
        Let \((X, \rho)\) be a connected metric space, which
        admits a connected dyadic cube decomposition
        (see \cref{def: connectedCubes}).
        Suppose that \(\mu\) is a finite positive Borel measure on \(X\)
        that is Ahlfors--David regular of dimension \(d\).
        Then, \(X\) admits a weighted decomposition 
        of order \(d\).
    \end{theorem}

    Since in an Ahlfors--David regular space the diameter of a set is
    bounded below in terms of its measure, the resulting bound on the
    diameter of each part is sharp up to constants (the formal statement
    and proof are given in \cref{cor: diameter}).

    The main extra condition which we require on the
    dyadic cube decomposition in order to apply this theorem is that
    the cubes are connected.
    Constructions developed in the abstract
    setting, for example in \cite{HK12},
    do not have this property in general.

    \smallskip

    To close this gap, we develop some tools
    that allow us to construct spaces which
    admit a connected dyadic cube decomposition.
    We shall call such a space \textbf{\cubedec}.
    Building \cubedec spaces is an interesting
    problem in its own right.
    One natural source of examples consists of
    metric spaces that are bi-Lipschitz images of
    Euclidean space, such as graphs of Lipschitz
    functions. However, not every manifold admits
    such a global bi-Lipschitz parametrisation.
    This motivates a triangulation-based approach:
    we obtain the well-studied family of
    metric spaces with a bi-Lipschitz triangulation.
    One such example is a compact \(C^1\) manifold. 
    It admits a bi-Lipschitz triangulation by the classical
    Cairns--Whitehead triangulation theorem
    \cite{Cai35, Whi40, Mun66}, combined with some additional 
    observations that we obtain in \cref{sec: manifolds}.

    \begin{theorem} \label{thm: manifoldIntro}
        Suppose that \(M\) is a compact connected
        \(C^1\) manifold, equipped with a Riemannian metric \(g\)
        and the corresponding geodesic metric \(\rho_g\),
        and \(\mu\) is an Ahlfors--David regular positive
        Borel measure on \(M\). Then, \(M\) admits
        a weighted decomposition.
    \end{theorem}

    \medskip

    \subsection*{Outline of the paper}
    In \cref{sec: dyadicCubes} we recall the basic definitions of Ahlfors--David regular
    spaces and of dyadic cubes, introduce the additional connectedness
    property (\cref{def: connectedCubes}) that we require
    throughout, and establish some technical properties of dyadic
    cubes.
    In \cref{sec: algorithm} we prove \cref{thm: partitionConnectedIntro},
    the main result of the paper: we construct an
    explicit algorithm that partitions \(X\) into pieces of prescribed
    measures.
    In \cref{sec: tools} we develop tools for verifying that a given space
    is \cubedec: in particular,
    we show that this property is preserved under bi-Lipschitz maps.
    We also show that such spaces may be glued together
    from finitely many pieces.
    In \cref{sec: manifolds} we apply these tools to obtain
    \cref{thm: manifoldIntro} and get further examples of
    well-known domains with the desired property.
    Finally, in \cref{sec: questions} we discuss open questions and
    applications of our result.

%% file: chapters/dyadicCubes.tex
\section{Basic definitions and dyadic cubes construction} \label{sec: dyadicCubes}

    In this paper we suppose that \(X\) is a connected metric space
    with metric \(\rho\). 
    The symbol \(\mu\) always denotes a positive Borel measure on \(X\).
    We write \(\overline{Y} \) for the closure of \(Y \subset X\)
    with respect to the topology of \(X\), where \(X\) is
    understood from the context (unless stated otherwise).
    Let \(B(x, R)\) denote the open ball with centre \(x\) and radius \(R\),
    where \(x\) is a point of some metric space \(X\).
    The notation \(\B^d\), \(\Delta^d\) and \(\sigma^d\) is used,
    respectively, for
    the open unit ball, the open regular simplex with unit
    edge length, and the open unit cube in \(\Rd\)
    centred at the origin,
    equipped with the standard Euclidean metric.
    
    We shall use the following notation throughout. For two
    nonnegative quantities \(A\) and \(B\) we write \(A \lesssim_X B\)
    if there exists a constant \(C_X > 0\), depending only on
    \(X\), such that \(A \le C_X B\). 
    We write \(A \gtrsim_X B\) if
    \(B \lesssim_X A\), and \(A \asymp_X B\) if both \(A \lesssim_X B\)
    and \(A \gtrsim_X B\) hold.

    \smallskip

    Recall the standard definitions of Ahlfors--David
    regularity and of a doubling metric
    space.

    \begin{definition} \label{def: doublingSpace}
        The metric space \(X\) is called \textbf{doubling}
        if there exists some constant \(C > 0\), such that
        for each \(x \in X\) and \(r > 0\)
        the ball \(B(x, 2r)\) may be covered by at most \(C\)
        balls of radius \(r\).
    \end{definition}

    \begin{definition} \label{def: AhlRegSpace}
        A metric space \(X\) with Borel measure \(\mu\)
        is called \textbf{Ahlfors--David regular}
        of dimension \(d\)
        if for each \(x \in X \) and 
        \(0 < r \le \mathrm{diam} X\) it holds that
        \[
            \mu( B(x, r) ) \asymp_X r^d.
        \]
    \end{definition}

    Note that if \(X\) may be equipped with some Borel measure \(\mu\),
    which makes it Ahlfors--David regular, then
    it is obviously doubling.
    Hence, we may assume
    that all metric spaces we work with are doubling.

    We will use the definition of dyadic cubes given below.

    \begin{definition} [\cite{HK12}, Thm. 2.2] \label{def: dyadicCubes}
        Suppose that \((X, \rho)\) is a metric space.
         The collection of subsets
        \(Q_\alpha ^ k \subset X \), where \(k \in \N_0 \) 
        and \(\alpha \in I_k \), is called
        a family of \textbf{dyadic cubes} if there exist
        a \textbf{ball constant} \(C > 0\) and
        a \textbf{scaling value} \(\delta \in (0, 1) \)
        such that these properties hold:
        \begin{itemize}
            \item [(a)] \(X = \bigsqcup_{\alpha \in I_k} Q_\alpha ^ k\)
            for each \(k \in \N_0\);
            \item [(b)] for each \(l \ge k\), \(\alpha \in I_l\) and \(\beta \in I_k\),
                either \(Q_\alpha ^ l \subset Q_\beta ^ k\)
                or \(Q_\alpha ^ l \cap Q_\beta ^ k = \emptyset \);
            \item [(c)] for each \(Q_\alpha ^ k \) there exists a point
                \(x_\alpha ^ k \in Q_\alpha ^ k\) such that
                \[
                    B(x_\alpha ^ k, C^{-1} \delta^k)
                    \subseteq
                    Q_\alpha ^ k
                    \subseteq
                    B(x_\alpha ^ k, C \delta^k).
                \]
        \end{itemize}
    \end{definition}

    This decomposition exists in quite general
    metric spaces.

    \begin{lemma} [\cite{HK12}, Thm. 2.2]
        \label{lem: decomposition}
        A doubling metric space admits a dyadic cube decomposition.
    \end{lemma}

    However, we need dyadic cubes to satisfy one more property,
    which makes the whole picture more complicated.

    \begin{definitionbold} \label{def: connectedCubes}
        A family of dyadic cubes \(Q_\alpha ^ k\) is called \textbf{connected}
        if properties (a)--(c) of \cref{def: dyadicCubes} hold,
        and, in addition:
        \begin{itemize}
            \item [(d)] each \(Q_\alpha ^ k\) is connected as a subset of \(X\).
        \end{itemize}
    \end{definitionbold}

    We now introduce the
    key definitions which we would like to work with.

    \begin{definitionbold} \label{def: cubeDecomp}
        Suppose that \((X, \rho)\) is a metric space.
        We say that it is \textbf{\cubedec} if it admits a
        connected dyadic cube decomposition.
        Suppose that \(Y \subset X\), equipped with the restricted
        metric, admits a connected dyadic cube decomposition
        \(Q_\alpha ^ k\).
        If, in addition, property (c) of \cref{def: dyadicCubes}
        holds with the inner ball taken in \(X\), we say that
        \(Y\) is \textbf{\cubedec in} \(X\).
    \end{definitionbold}

    \begin{remark}
        In particular, a set that is \cubedec in \(X\) is
        \cubedec, and \(X\) is \cubedec in itself.
        The converse does not hold. For example, consider 
        a circle \(\T \subset \R^2\) with the restricted 
        Euclidean metric.
        It is trivial that \(\T\) is \cubedec, but no ball
        lies inside it, so \(\T\) is not \cubedec in \(\R^2\).
    \end{remark}

    Next, we introduce some notation and
    prove extra properties of dyadic cubes that we also need.

    Say that \(\eta \subset \sigma \) for some \(\eta \in I_l \), 
    \(\sigma \in I_k\) and
    \(l \ge k \) if \(Q_\eta ^ l \subset Q_\sigma ^ k \).
    We say that \(\eta\) is a \textbf{nested descendant}
    of \(\sigma\) and
    \(\sigma\) is a \textbf{nested ancestor} of \(\eta\), if
    \(\eta \subset \sigma \).
    If \(\eta \in I_{k + 1}\) and \(\sigma \in I_{k} \),
    we call them \textbf{nested child} and \textbf{nested parent}
    respectively.
    It follows from property (b) of \cref{def: dyadicCubes} that each cube
    of generation \(k \ge 1\) has a unique nested parent.
    For \(0 \le k \le l \) and \(\sigma \in I_k\), denote
    \[ 
        I_{l, \sigma} 
        := 
        \{ \eta \subset \sigma \mid \eta \in I_l\}.
    \]
    Property (b) of \cref{def: dyadicCubes} gives that
    \begin{equation} \label{eq: childUnion}
        Q_\sigma ^ k 
        = 
        \bigsqcup_{\eta \in I_{l, \sigma}} Q_\eta ^ l.
    \end{equation}

    \smallskip

    Given a family of dyadic cubes, we may consider
    a well-ordering \(\prec_k \) of each \(I_k\).
    We construct it inductively.
    The order on \(I_0\) is arbitrary.
    Let \(k \ge 0\) and \(\alpha, \beta \in I_{k + 1}\).
    Consider nested parents \(\alpha', \beta' \in I_k\),
    \(\alpha \subset \alpha'\) and \(\beta \subset \beta'\).
    If \(\alpha' \prec_k \beta'\) we say that
    \(\alpha \prec_{k+1} \beta\). This gives an order
    on pairs with different parents.
    On the children of some \(\sigma \in I_k\)
    we construct an arbitrary well-ordering.

    \smallskip 

    We will state some properties of the
    closure of dyadic cubes.
    Firstly, we show that in a doubling metric
    space the decomposition is locally finite.

    \begin{proposition} \label{prop: finite}
        Suppose that \(X\) is doubling,
        and some \(Y \subset X\) admits a
        dyadic cube decomposition \(Q_\alpha ^ k\).
        Then, for each \(k \in \N_0\), the family
        \(\{Q_\alpha ^ k\}_{\alpha \in I_k}\)
        is locally finite.
    \end{proposition}

    \begin{proof}
        Take some \(x \in X\) and consider the neighbourhood
        \(U := B(x, C \delta ^ k) \).
        If \(U \cap Q_\alpha ^ k \neq \emptyset\),
        we have
        \[
           B(x_\alpha ^ k, C^{-1} \delta^k)
           \subset
           Q_\alpha ^ k
           \subset
           B(x, 3 C \delta ^ k),
        \]
        and the balls \(B(x_\alpha ^ k, C^{-1} \delta^k)\) are
        pairwise disjoint for distinct \(\alpha\), since the cubes
        \(Q_\alpha ^ k\) are.
        The doubling property applied to \(B(x, 3C \delta^k)\)
        implies that only finitely many pairwise disjoint balls of
        radius \(C^{-1} \delta^k\) fit inside it, so only finitely
        many such \(\alpha\) exist, and the family is indeed
        locally finite.
    \end{proof}

    Now, we conclude with a property of closures.

    \begin{corollary} \label{cor: closureUnion}
        Suppose that \(X\) is doubling,
        and some \(Y \subset X\) admits a
        dyadic cube decomposition \(Q_\alpha ^ k\).
        Then,
        \begin{equation} \label{eq: closureChildUnion}
            \overline{Q}_{\alpha'} ^ k
            =
            \bigcup_{\alpha \in I_{k + 1, \alpha'}} \overline{Q}_\alpha ^ {k + 1}
        \end{equation}
        and
        \begin{equation} \label{eq: closureUnion}
            \overline{Y} = \bigcup_{\alpha \in I_k} \overline{Q}_\alpha ^ k.
        \end{equation}
    \end{corollary}

    \begin{proof}
        By \cref{prop: finite}, the family \(\{Q_\alpha ^ k\}_{\alpha \in I_k}\)
        is locally finite for each \(k\).
        Then, \eqref{eq: childUnion} yields
        \[
            Q_{\alpha'} ^ k
            =
            \bigsqcup_{\alpha \in I_{k + 1, \alpha'}} Q_\alpha ^ {k + 1}.
        \]
        The standard topological fact that
        closure commutes with locally finite unions,
        gives \eqref{eq: closureChildUnion}.
        The same fact, applied
        to property (a) of \cref{def: dyadicCubes}, gives
        \eqref{eq: closureUnion}, and the corollary follows.
    \end{proof}

    The following technical property 
    allows us to transfer connected decompositions 
    to the closures.

    \begin{proposition} \label{prop: sandwich}
        Suppose that \(Y \subset X\) is \cubedec
        in \(X\), the closure \(\overline{Y}\) with the restricted
        metric is doubling, and \(Y \subset Z \subset \overline{Y}\).
        Then \(Z\) is also \cubedec in \(X\).
    \end{proposition}

    \begin{proof}
        Consider a dyadic cube decomposition
        \(\mathcal{P} = \{ P_\alpha ^ k \mid k \in \N_0, \alpha \in I_k \}\)
        for \(Y\).
        By \eqref{eq: closureUnion},
        \(Z \subset \overline{Y} = \bigcup_{\alpha \in I_k} \overline{P}_\alpha ^ k\).

        \smallskip

        We will build a decomposition
        \(\mathcal{Q} = \{ Q_\alpha ^ k \mid k \in \N_0, \alpha \in I_k \}\)
        for \(Z\).
        Fix some \(x \in Z\).  
        If \(x \in Y\), attach \(x\) to \(Q_\alpha ^ k\), where \(\alpha\)
        is the unique index with \(x \in P_\alpha ^ k\).
        Suppose that \(x \in Z \setminus Y \).
        By \cref{prop: finite} applied to \(\overline{Y}\), 
        only finitely many \(\eta \in I_k\)
        satisfy \(x \in \overline{P}_\eta ^ k\), so a \(\prec_k\)-minimal
        such \(\alpha\) exists.
        Attach \(x\) to \(Q_\alpha ^ k\).
        By construction \(\{Q_\alpha^k\}\) partitions \(Z\) at each level \(k\),
        giving property (a) of \cref{def: connectedCubes}.

        \smallskip

        To check (b), we take some
        \(\alpha \in I_{k + 1}\) and consider its parent \(\alpha' \in I_k\).
        We prove that \(Q_\alpha^{k+1} \subseteq Q_{\alpha'}^k\).
        Consider some \(x \in Q_\alpha^{k+1}\).

        If \(x \in Y\), then our construction gives
        \(Q_\alpha^{k+1} \cap Y = P_\alpha^{k+1} \cap Y\), hence
        \(x \in P_\alpha^{k+1}\).
        Thus property (b) of \(\mathcal{P}\) gives
        \(x \in P_\alpha^{k+1} \subset P_{\alpha'}^k \subset Q_{\alpha'} ^ k\).

        For \(x \notin Y\): by \eqref{eq: closureChildUnion},
        \(\overline{P}_{\alpha'}^k = \bigcup_{\alpha \in I_{k + 1, \alpha'}} \overline{P}_\alpha^{k+1}\).
        By definition of \(\prec_{k+1}\), the \(\prec_{k+1}\)-minimal
        eligible \(\alpha\) is always a child of the \(\prec_k\)-minimal eligible
        \(\alpha'\), so \(x \in Q_{\alpha'}^k\) here too.
        In either case, together with (a) this gives (b) by iterating over levels.

        \smallskip

        Inclusions
        \[
            B(x_\alpha ^ k, C^{-1} \delta ^ k)
            \subset
            P_\alpha ^ k
            \subset
            Q_\alpha ^ k
            \subset
            \overline{P}_\alpha ^ k
            \subset
            B(x_\alpha ^ k, 2C \delta^k)
        \]
        (where balls are taken in \(X\)) follow from (c).
        Since \(P_\alpha ^ k \subset Q_\alpha ^ k \subset \overline{P}_\alpha ^ k\)
        and \(P_\alpha ^ k\) is connected, so is \(Q_\alpha ^ k\), 
        which gives (d) for \(\mathcal{Q}\).
    \end{proof}

    \smallskip

    It turns out that if we have a decomposition, 
    we obtain decompositions with every possible 
    scaling value.

    \begin{proposition} \label{prop: scaling}
        Suppose that \(\delta, \theta \in (0, 1)\), and \(Q_\alpha ^ k\)
        is a family of dyadic cubes with the scaling value \(\delta\).
        Then there exists a nondecreasing sequence
        \(\{m_k\}_{k \in \N_0} \), with \(m_0 = 0\), such that
        \(\{Q_\alpha ^ {m_k} \mid \alpha \in I_{m_k}\}_{k \in \N_0}\)
        is also a family of dyadic cubes with the scaling value \(\theta\).
    \end{proposition}

    \begin{proof}
        Let \(r := \dfrac{\log \theta}{\log \delta} > 0\)
        and set \(m_k := \lfloor k r \rfloor\).
        Then \(m_0 = 0\), and \((m_k)\) is nondecreasing since \(r > 0\).

        Properties (a), (b) of \cref{def: connectedCubes}
        for \(\{Q_\alpha ^ {m_k}\}\) follow immediately.

        For (c), since \(-1 < m_k - rk \le 0 \), we have
        \(\theta^k \le \delta^{m_k} < \delta^{-1} \theta^k\), 
        and combining this with (c) gives
        \[
            B(x_\alpha ^{m_k}, C^{-1} \theta^k)
            \subseteq
            Q_\alpha ^{m_k}
            \subseteq
            B(x_\alpha ^{m_k}, (C / \delta) \, \theta^k).
        \]
        Consequently,
        \(\{Q_\alpha ^{m_k}\}\) satisfies (c) with scaling value \(\theta\)
        and ball constant \(C / \delta\).
    \end{proof}

    Note that if the original dyadic cubes were connected,
    the renumbered family remains connected as well.
    We formulate a direct consequence of this result below.

    \begin{corollary} \label{cor: goodScaling}
        Suppose that \(Y\) is \cubedec in \(X\) 
        with scaling value \(\delta\).
        Then, it is \cubedec in \(X\) 
        for each scaling value.
    \end{corollary}

    \begin{proof}
        It immediately follows from \cref{prop: scaling}.
    \end{proof}

    \medskip

    The following two simple technical 
    observations are taken from \cite{GL17}.

    \begin{proposition} [\cite{GL17}, Cor. 2] \label{prop: cubes}
        Suppose that \(\overline{Q}_\alpha ^ k \cap \overline{Q}_\beta ^ k \neq \emptyset\).
        Then,
        \[
            Q_\beta ^ k \subset B(x_\alpha ^ k, 3 C \delta^k).
        \]
    \end{proposition}

    \begin{proposition} [\cite{GL17}, Cor. 3] \label{prop: cutMass}
        Suppose that \(X\) is an Ahlfors--David regular metric measure space,
        and \(S \subset X\) is a measurable subset with finite measure.
        Then, for each \(0 \le t \le \mu(S) \) there exists
        \(T \subseteq S\) such that \(\mu(T) = t\).
    \end{proposition}

%% file: chapters/algorithm.tex
\section{The decomposition algorithm} \label{sec: algorithm}

    In this section we prove \cref{thm: partitionConnectedIntro},
    which provides a technical condition on a
    metric space under which it admits a partition with
    prescribed measures.
    The construction is algorithmic.

    \begin{theorem} \label{thm: partitionConnected}
        Let \((X, \rho)\) be a connected metric space, which
        is \cubedec (\cref{def: cubeDecomp}).
        Suppose that \(\mu\) is a finite positive
        Borel measure on \(X\) that is Ahlfors--David regular
        of dimension \(d\).
        Then, there exists a constant \(C^* > 0\),
        depending only on \((X, \rho, \mu)\), such
        that the following holds.
        For all positive numbers
        \(\lambda_1, \dots, \lambda_n\) with
        \(\sum_{j=1}^n \lambda_j = \mu(X)\), there exist measurable subsets
        \(X_1, \dots, X_n \subset X\)
        such that \(\bigsqcup_{j=1}^n X_j = X, \, \mu(X_j)= \lambda_j,\)
        and each \(X_j\) is contained in a ball of radius
        \(C^* \lambda_j^{\frac{1}{d}}\).
    \end{theorem}

    We first observe that the estimate on the
    diameters of \(X_j\) is sharp up to a constant,
    as announced in the introduction.

    \begin{corollary} \label{cor: diameter}
        In the notation of
        \cref{thm: partitionConnected}, 
        \(\diam X_j \asymp_X \lambda_j ^ {1/d} \).
    \end{corollary}

    \begin{proof}
        Observe that \(X_j \subset B(x, 2\diam X_j)\) 
        for some \(x \in X_j\).
        Consequently, since the measure is Ahlfors--David regular,
        \[ 
            \lambda_j = \mu(X_j)
            \lesssim_X
            (\diam X_j )^d.
        \]
        Hence, \(\diam X_j \gtrsim_X \lambda_j ^ {\frac{1}{d}}\).
        Since the weighted decomposition gives the upper bound
        \(\diam X_j \lesssim_X \lambda_j ^ {\frac{1}{d}}\),
        the statement follows.
    \end{proof}

    We now prove the main theorem.

    \begin{proof} [Proof of \cref{thm: partitionConnected}]
        Since \(\mu\) is Ahlfors--David regular, \(\mu(X) > 0\).
        Dividing the numbers \(\lambda_j\) by \(\mu(X)\), we
        may assume that \(\mu(X) = 1\).
        Fix some family of connected dyadic cubes \(Q_\alpha ^ k \)
        with ball constant \(C\).
        Since \(\mu(X) < \infty\) and \(\mu\) is Ahlfors--David regular,
        \(X\) is bounded.
        Without loss of generality, we may therefore assume that
        \(|I_0| = 1\), taking \(Q_\alpha ^ 0 := X\) and enlarging the
        ball constant to be greater than \(\diam X\) if necessary.
        Denote by \(\overline{Q}_\alpha ^ k\) the closure of \(Q_\alpha ^ k\).
        Let \(M_\alpha ^ k := \mu(Q_\alpha ^ k) \).
        Using property (c) of \cref{def: dyadicCubes} and
        Ahlfors--David regularity, we obtain
        \[
            M_\alpha ^ k
            =
            \mu(Q_\alpha ^ k ) 
            \ge 
            \mu(B(x_\alpha ^ k, C^{-1} \delta^k))
            \gtrsim_X
            (C^{-1} \delta ^ k)^d
            \gtrsim_X
            \delta^{kd}
        \]
        for each \(\alpha \in I_k\), and \(M_\alpha ^ k \lesssim_X \delta^{kd}\)
        as well.
        In other words, 
        \(s_k := a_0 \delta^{kd} \le M_\alpha ^ k \le a_1 \delta^{kd} =: S_k \) 
        for some 
        \(a_0, a_1 > 0\), which depend only on \(X\). 
        Since \(\mu(X) = 1\), it follows
        that \(I_k\) is finite.

        \smallskip

        We create a graph \(G_k\) with the vertices \(V(G_k) := I_k\)
        associated with the cubes \(Q_\alpha ^ k \).
        We say that \((\alpha, \beta) \in E(G_k) \) if
        \(\overline{Q}_\alpha ^ k \cap \overline{Q}_\beta ^ k \neq \emptyset \).

        \smallskip

        Let \(\sigma \in I_k\) and \(l \ge k\).
        Denote by \(G_{l, \sigma}\) the subgraph of \(G_l\) on
        the vertices of \(I_{l, \sigma}\).

        \begin{proposition} \label{prop: connected}
            For every \(k \ge 1 \) and \(\sigma \in I_{k - 1}\), 
            the graph \(G_{k, \sigma}\) is connected.
        \end{proposition}

        \begin{proof}
            Suppose that \(G_{k, \sigma}\) is disconnected:
            \(I_{k, \sigma} = I ^ 1 \sqcup I ^ 2\), both nonempty,
            with no edges between \(I ^ 1\) and \(I ^ 2\). Set
            \(Y_1 = \bigcup_{\eta \in I ^ 1} \overline{Q}_\eta ^ k\) and
            \(Y_2 = \bigcup_{\eta \in I ^ 2} \overline{Q}_\eta ^ k\). By assumption,
            \(Y_1 \cap Y_2 = \emptyset\).
            Since these unions are finite, both sets are closed.
            Hence, \(\overline{Q}_\sigma ^ {k-1} = Y_1 \sqcup Y_2\) by
            \eqref{eq: closureChildUnion}, splitting it into two disjoint
            nonempty closed sets --- contradicting the fact that
            \(Q_\sigma ^ {k-1}\), and hence its closure, is connected.
        \end{proof}

        We now construct, recursively, the oriented spanning tree \(T_k\)
        of \(G_k\), such that all edges of \(T_k\) are oriented away from
        a single root vertex.

        The tree \(T_0\) consists of the single vertex of \(I_0\).
        Suppose that \(T_k\) is already
        constructed.
        Consider some \(\beta' \in I_k\).
        If it is not the root of the tree, take the unique edge
        \(\alpha' \to \beta'\) in \(T_k\).
        Consider all pairs \(\alpha, \beta \in I_{k + 1}\)
        such that \(\alpha \subset \alpha'\) and \(\beta \subset \beta'\).
        At least one such pair belongs to \(E(G_{k+1})\): take
        \(x \in \overline{Q}_{\alpha'} ^ k \cap \overline{Q}_{\beta'} ^ k\). 
        By \eqref{eq: closureChildUnion}, \(x\) lies in
        \(\overline{Q}_\alpha ^ {k+1} \cap \overline{Q}_\beta ^ {k+1}\) for
        some such \(\alpha, \beta\).
        We take one such pair and orient \((\alpha \to \beta)\).
        Then, we choose a spanning tree of \(G_{k+1, \beta'}\) (it exists by
        \cref{prop: connected}) and orient it away from the root \(\beta\).
        If \(\beta'\) is the root of \(T_k\), we instead choose a spanning
        tree of \(G_{k+1, \beta'}\) and orient it away from an arbitrary root.

        \input{figures/GkTk.tex}

        \smallskip

        We say that \(\beta\) is a \textbf{child} of
        \(\alpha\), and \(\alpha\) is a \textbf{parent} of
        \(\beta\) in \(T_k\) if the
        edge \((\alpha, \beta) \in E(T_k)\)
        is oriented \(\alpha \to \beta\).
        If there exists a directed path
        \(\alpha = \gamma_0 \to \dots \to \gamma_m = \beta\) in \(T_k\),
        we call \(\alpha\) an \textbf{ancestor} of \(\beta\)
        and \(\beta\) a \textbf{descendant} of \(\alpha\) in \(T_k\).

        \begin{remark}
            This is not related to the \textbf{nested} child/parent
            relation of \cref{sec: dyadicCubes}; we always write ``nested'' when
            we mean that one.
        \end{remark}

        \smallskip

        One may think of \(G_k\) as a system of connected containers, which we shall
        fill with water. We will give some intuitive definitions.

        Take \(K\) such that \(S_K \le \min_j \lambda_j \).
        We shall consider only generations of cubes
        with \(k \le K \).
        Consider some \textbf{configuration}: a measurable subset \(V \subset X \).
        We denote \(m_\alpha ^ k := \mu(V \cap Q_\alpha ^ k) \) and
        \(r_\alpha ^ k := M_\alpha ^ k - m_\alpha ^ k \) for all \(\alpha \in I_k \).
        These numbers represent the amount of water collected
        in the cubes and the volume of the remaining space, respectively.
        Since \(Q_\alpha ^ k = \bigsqcup_{\beta \in I_{l, \alpha}} Q_\beta ^ l\)
        for \(l \ge k\) by \eqref{eq: childUnion},
        we infer the following properties:
        \begin{equation} \label{eq: additive}
            M_\alpha ^ k =
            \sum_{\beta \in I_{l, \alpha}} M_\beta ^ l,
            \qquad
            m_\alpha ^ k =
            \sum_{\beta \in I_{l, \alpha}} m_\beta ^ l,
            \qquad
            r_\alpha ^ k =
            \sum_{\beta \in I_{l, \alpha}} r_\beta ^ l.
        \end{equation}
        The vertex \(\alpha \in I_k\) is:
        \begin{itemize}
            \item \textbf{empty} if \(m_\alpha ^ k = 0\);
            \item \textbf{full} if \(r_\alpha ^ k = 0\).
        \end{itemize}
        Since \(M_\alpha ^ k \ge s_k > 0\), no vertex is both empty and full.
        Obviously, when the configuration grows, nonempty vertices stay nonempty
        and full ones stay full.
        By \eqref{eq: additive}, a vertex is full (respectively, empty)
        if and only if all of its nested children are.
        We shall also say that \(\alpha\) is \textbf{ripe} if it is not full, but its
        children in \(T_k\) are.
        Suppose that \(0 \le k \le K\).
        The configuration \(V\) is called \(k\)-\textbf{legal} if every nonempty
        cube of generation \(l\), where \(k \le l \le K\),
        has all full children in \(T_l\).
        Note that if the configuration is \(k\)-legal, it is also \(l\)-legal
        for all \(k \le l \le K\).
        Moreover, in this case each cube of generation \(l\), where
        \(k \le l \le K\), is empty, full or ripe (a nonempty cube has all
        full children automatically).
        The fact that the configuration is \(k\)-legal means that the water,
        considered in \(T_k, \dots, T_K\), is in
        a state of hydrostatic equilibrium.
        
        \input{figures/configuration.tex}

        Let us formulate a simple observation about the ripe vertices.

        \begin{proposition} \label{prop: ripeNestChild}
            Suppose that \(k < K\).
            Then, for each ripe \(\sigma \in I_k\), there exists a ripe
            \(\eta \in I_{k + 1}\), such that \(\eta \subset \sigma\).
        \end{proposition}

        \begin{proof}
            Since \(\sigma\) is not full, by \eqref{eq: additive} there
            exists some nonfull
            nested child of \(\sigma\). Descending from it in the finite
            spanning tree of \(G_{k + 1, \sigma}\), we find, in this tree,
            such a child without any nonfull children; denote it \(\eta\).
            Suppose that \(\zeta\) is a child of \(\eta\).
            If \(\zeta \subset \sigma\), it is full due to the choice of \(\eta\).
            Otherwise, consider its nested parent \(\zeta'\).
            By construction of \(T_{k + 1}\), \( \sigma \to \zeta'\) in
            \(T_k\) and \(\zeta'\) is full. Hence \(\zeta\) is full too
            by \eqref{eq: additive}.
            As a result, \(\eta\) is ripe.
        \end{proof}

        \smallskip

        Now, we define the important procedure of the algorithm that we call
        \textit{saturation}.
        Fix a system with the current configuration \(V\).
        The saturation of mass \(t \ge 0\) may be applied to a
        ripe or full vertex \(\alpha \in I_k\).
        We denote it \(\mathcal{S}(\alpha, t) \). The result of this procedure
        is the set \(W \subset X \setminus V\). 
        Intuitively, this is the set that got sealed off by the inserted water.
        The new configuration \(V' := V \sqcup W \) is constructed automatically.
        We define the saturation by downward induction on the level \(k\).

        \begin{itemize}

            \item[(1)] Firstly, we define the saturation at a ripe (or full)
            \(\eta \in I_K \). The saturation of the mass \(t \) applied to \(\eta\) 
            means taking the subset \(W \subset Q_\eta ^ K \setminus V \) such that
            \[
                \mu(W) = \min(r_\eta ^ K, t).
            \]
            That is possible due to \cref{prop: cutMass}. 
            If \(r_\eta ^ K \le t \), we take
            \(W = Q_\eta ^ K \setminus V \) exactly.
            We set \(\mathcal{S}(\eta, t) := W \).

            \item[(2)] Now consider a ripe (or full) \(\alpha \in I_k\)
            with arbitrary \(0 \le k \le K\).
            If \(k = K\), apply step (1). Suppose that \(k < K\).
            We want to define \(\mathcal{S}(\alpha, t) \).
            Set \(W := \emptyset\) and \(z := t\).
            While \(z > 0\) and \(\alpha\) is not full, 
            take some ripe nested child
            \(\beta \in I_{k+1}\), \(\beta \subset \alpha\).
            Such a child exists due to \cref{prop: ripeNestChild},
            since all children of \(\alpha\) remain full during the 
            procedure, so \(\alpha\) remains ripe.
            Then, execute
            \[
                W_\beta := \mathcal{S}(\beta, z),
            \]
            and take \(W := W \sqcup W_\beta\) and \(z := z - \mu(W_\beta)\).
            We return \(W\).
        \end{itemize}

        \input{figures/saturation.tex}

        Let us formulate some simple properties of the saturation.

        \begin{proposition} \label{prop: saturation}
            Suppose that \(V\) is \(k\)-legal, \(\alpha \in I_k\)
            is ripe or full,
            \(k \le K\), \(t \ge 0\), and
            \(\mathcal{S}(\alpha, t) = W\). Then:
            \begin{itemize}
                \item [(a)] \(W \subset Q_\alpha ^ k \setminus V\);
                \item [(b)] \(\mu(W) = \min(r_\alpha ^ k, t)\);
                \item [(c)] \(V \sqcup W\) is \(k\)-legal.
            \end{itemize}
        \end{proposition}

        \begin{proof}
            Induction on the definition of \(\mathcal{S}\).

            In case (1) everything is immediate from the construction:
            \(\mu(W) = \min(r_\eta ^ K, t)\) gives (b), and the set \(W\)
            lies in \(Q_\eta ^ K \setminus V\), which gives (a).
            Since we changed only a ripe (or full)
            vertex, the configuration stays \(K\)-legal.

            In case (2), \(W\) is a disjoint union of
            pieces \(W_\beta\) from recursive calls \(\mathcal{S}(\beta, \cdot)\),
            with \(\beta \subset \alpha\).
            Each intermediate configuration is \((k + 1)\)-legal by induction,
            so the proposition applies to every recursive call.
            Note that an iteration leaving a positive remaining mass has
            \(\mu(W_\beta) = \min(r_\beta ^ {k + 1}, z) = r_\beta ^ {k + 1}\)
            by (b), so \(\beta\) becomes full, and it is never chosen again.
            Thus, the loop terminates.
            By induction,
            \(W_\beta \subset Q_\beta ^ {k + 1} \setminus V \subset
            Q_\alpha ^ k \setminus V\) and (a) follows.
            If the process finished with \(z = t_0\), the construction gives
            \(\mu(W) = \sum_\beta \mu(W_\beta) = t - t_0\).
            If \(t_0 = 0\), then \(\mu(W) = t\), and \(t \le r_\alpha ^ k\)
            by (a).
            Otherwise, the loop stops when
            \(\alpha\) is full. Hence \(\mu(W) = r_\alpha ^ k = t - t_0 \le t\),
            which proves (b).
            To check (c), observe that (a) gives \(W \subset Q_\alpha ^ k\).
            Therefore, on level \(k\) the procedure changes the status of
            \(\alpha\) only, and \(\alpha\) is either ripe or full,
            so the legality condition on level \(k\) is satisfied.
            Induction gives that all recursive calls keep the configuration
            \((k + 1)\)-legal. Consequently, the result is \(k\)-legal.
        \end{proof}

        In terms of our hydrostatic model, we just add a portion of water
        with mass \(t\) to the vertex \(\alpha\) and allow it to leak down.

        We will also introduce the notion of the \textit{canonical saturation}
        of mass \(w\) or \(\mathcal{CS}(w) \).
        Assume that \(S_K \le w \le \mu(X \setminus V) \)
        (there is enough space left, and the mass is not too small).
        Take the maximal \(k \in \N_0\) such that \(w \le s_k \), and \(k = 0\)
        if there is no such \(k\).
        Since \(w \ge S_K > s_{K + 1}\), we have \(k \le K\).
        Suppose that the configuration is \(k\)-legal.

        Consider the following
        dyadic cube \(\alpha \in I_k\): if there are any empty cubes
        in \(I_k\), take one without empty children in \(T_k\)
        (descending in \(T_k\)).
        Otherwise, take the root of \(T_k\).
        In either case, all children of \(\alpha\) in \(T_k\) are 
        ripe or full.
        Now execute the following sequence of saturations.
        Put \(z := w\) and \(W_0 := \emptyset\).
        While \(z > 0\) and \(\alpha\) has a nonfull child in \(T_k\),
        choose such a child \(\beta\), let \(W_\beta := \mathcal{S}(\beta, z) \),
        and take \(W_0 := W_0 \sqcup W_\beta \) and \(z := z - \mu(W_\beta)\).
        Since each saturation makes \(\beta\) full or \(z = 0\),
        no child can be picked twice, so the loop finishes.
        Consider both possible endings.
        \begin{itemize}
            \item [(1)] The loop stops since \(z = 0\).
            Then, put \(W := W_0\).
            \item [(2)] We stop since all children
            are full, hence \(\alpha\) is ripe (or full) and we
            may apply saturation to it.
            Put \(W := W_0 \sqcup W_\alpha\),
            where \(W_\alpha := \mathcal{S}(\alpha, z)\).
        \end{itemize}
        
        Finally, we set \(\mathcal{CS}(w) := W\).

        \input{figures/canonSaturation.tex}

        \begin{lemma} \label{lem: main}
            Let \(r := w^{\frac{1}{d}}\). 
            In the notation above we have:
            \begin{itemize}
                \item [(a)] \(V \sqcup W\) is \(k\)-legal;
                \item [(b)] \(\mu(W) = w\);
                \item [(c)] there exist a constant \(C^* > 0\),
                depending only on \(X, \rho, \mu\), and a point
                \(x \in X\) such that \( W \subset B(x, C^* r)\).
            \end{itemize}
        \end{lemma}

        \begin{proof}
            Let \(\alpha \in I_k\) be as in the
            definition of \(\mathcal{CS}(w)\), 
            and set \(x := x_\alpha ^ k\).

            By \cref{prop: saturation}, each step of the canonical
            saturation preserves the
            \(k\)-legality of the configuration, and (a) follows.
            If \(\alpha\) is empty, we know that it does not have empty children,
            and \(r_\alpha ^ k = M_\alpha ^ k \ge w\).
            Indeed, \(M_\alpha ^ k \ge s_k \ge w\) by the choice of \(k\).
            In case \(k = 0\) we have \(M_\alpha ^ 0 = \mu(X) = 1 \ge w\).
            By \cref{prop: saturation} (a) and the construction of \(W\),
            the set \(W\) lies in the union of the cubes associated with
            \(\alpha\) and its children in \(T_k\).
            Thus,
            \[
                W
                \subset
                \left( \bigcup_{\beta \,:\, (\alpha, \beta) \in E(G_k)} 
                Q_\beta ^ k \right) \cup Q_\alpha ^ k
                \subset
                B(x_\alpha ^ k, 3C \delta^k)
            \]
            (the last inclusion holds due to \cref{prop: cubes}).

            If there are no empty cubes in \(I_k\), the root is not empty,
            and, consequently, all other cubes are full
            (as descendants of a nonempty cube in a \(k\)-legal configuration),
            so \(W \subset Q_\alpha ^ k \subset B(x_\alpha ^ k, C \delta^k)\)
            as well. In this case \(r_\alpha ^ k = \mu(X \setminus V) \ge w\).

            In either case, \(W \subset B(x_\alpha ^ k, 3C \delta^k)\) and
            \(r_\alpha ^ k \ge w\).
            
            It remains to show that \(\mu(W) = w\).
            If we do not call \(\alpha\) in the canonical saturation,
            \(\mu(W) = \mu(W_0) = w\).
            Otherwise, the call \(\mathcal{S}(\alpha, z)\) is executed 
            with some \(z \le w \le r_\alpha ^ k\). 
            Hence \cref{prop: saturation} (b) gives \(\mu(W_\alpha) = z\),
            so \(\mu(W) = \mu(W_0) + \mu(W_\alpha) = (w - z) + z = w\).
            Property (b) is proved.

            The choice of \(k\) guarantees that
            \(w \ge s_{k+1} = a_0 \delta ^{(k+1)d} \gtrsim_X \delta^{kd}\).
            Consequently,
            \[
                3C \delta^k
                \lesssim_X
                \delta^{k}
                \lesssim_X w^{1/d} = r,
            \]
            which gives (c) for a suitable constant \(C^*\).
        \end{proof}

        Without loss of generality, we may rearrange masses and suppose that 
        \(\lambda_1 \ge \lambda_2 \ge \dots \ge \lambda_n\).
        The final version of the algorithm is the sequence of
        canonical saturations \(\mathcal{CS}(\lambda_j) \) 
        for \(j = 1, \dots, n \). The results of these steps are sets \(X_j\).
        Denote by \(k_j\) the generation chosen in the definition of
        \(\mathcal{CS}(\lambda_j)\). Since the masses decrease,
        \(k_1 \le \dots \le k_n\). 
        The empty configuration is \(k_1\)-legal, and induction, combined
        with \cref{lem: main}, gives that before step \(j\) the
        configuration is \(k_j\)-legal.
        Moreover,
        \(\mu(X \setminus V) = \sum_{i \ge j} \lambda_i \ge \lambda_j\),
        while the choice of \(K\) gives \(S_K \le \lambda_j\),
        so \cref{lem: main} applies at every step and yields the desired
        ball bound on \(X_j\).
        The saturation procedure guarantees that the sets \(X_j\) are
        pairwise disjoint.
        Denote \(X' := \bigcup_{j = 1} ^ n X_j \subset X\).
        Since \(\mu(X') = \sum_{j = 1}^n \mu(X_j) = 1\),
        \(\mu(X \setminus X') = 0\).
        Therefore, all the numbers \(r_\alpha ^ K\), \(\alpha \in I_K\),
        vanish.
        Whenever a saturation step reduces some \(r_\alpha ^ K\) to \(0\),
        it sets \(W = Q_\alpha ^ K \setminus V\) exactly, 
        so \(Q_\alpha ^ K \subset X'\).
        Taking the union over \(\alpha \in I_K\) gives \(X \subset X'\),
        so \(X = X'\).
        The proof is finished.
    \end{proof}

%% file: figures/GkTk.tex
\begin{figure}[htbp]
    \begin{minipage}[c]{0.52\textwidth}
    \centering
    \includegraphics[width=\linewidth]{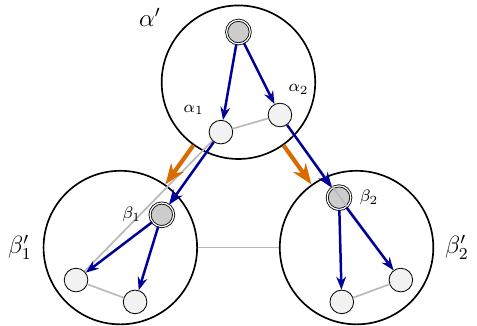}
    \end{minipage}%
    \hfill
    \begin{minipage}[c]{0.43\textwidth}
    \raggedright
    \caption
    {
        The recursive construction of \(T_k\) (blue)
        from \(T_{k-1}\) (orange).
        \label{fig: GkTk}
    }
    \end{minipage}
\end{figure}

%% file: figures/configuration.tex
\begin{figure}[htbp]
    \centering
    \includegraphics{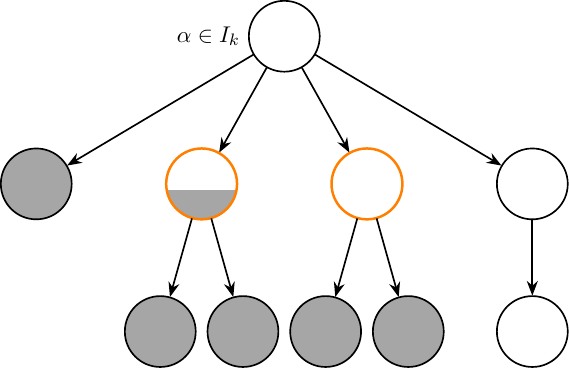}
    \caption{
        A \(k\)-legal configuration \(V\) and the children of
        \(\alpha \in I_k\):
        full, empty, and ripe (orange).
     \label{fig: configuration}
    }
\end{figure}

%% file: figures/saturation.tex
\begin{figure}[htbp]
    \centering
    \includegraphics{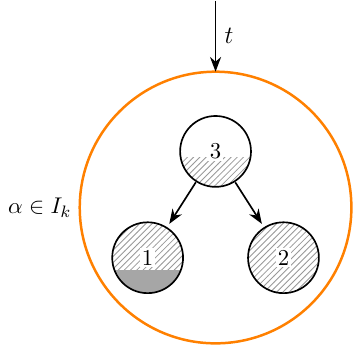}
    \caption{
        The saturation \(\mathcal{S}(\alpha, t)\) applied to a ripe
        \(\alpha \in I_k\).
        The numbers give the order in which the nested 
        children are saturated.
    }
    \label{fig: saturation}
\end{figure}

%% file: figures/canonSaturation.tex
\begin{figure}[htbp]
    \centering
    \captionsetup[subfigure]{labelformat=empty}
    \begin{subfigure}[b]{0.48\textwidth}
        \centering
        \includegraphics{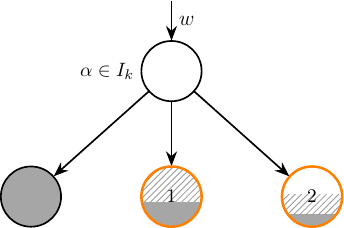}
        \caption{
            Case (1): the loop ends with \(z = 0\).
        }
        \label{fig: canonSaturationA}
    \end{subfigure}
    \hfill
    \begin{subfigure}[b]{0.48\textwidth}
        \centering
        \includegraphics{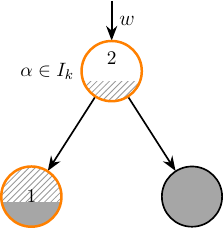}
        \caption{
            Case (2): all children of \(\alpha\) become full.
        }
        \label{fig: canonSaturationB}
    \end{subfigure}
    \caption{
        The canonical saturation \(\mathcal{CS}(w)\) at a cube
        \(\alpha \in I_k\) without empty children, in the two
        possible cases (1) and (2).
    }
    \label{fig: canonSaturation}
\end{figure}

%% file: chapters/tools.tex
\section{Technical tools} \label{sec: tools}

    The goal of the current section is to develop some
    tools to construct spaces with a connected dyadic cube decomposition
    (recall that we call these spaces \cubedec), and then
    apply \cref{thm: partitionConnected}.

    \smallskip

    We begin with the simplest \cubedec space.
    
    \begin{proposition} \label{prop: cubeGood}
        Euclidean space \(\Rd\) is \cubedec.
        Standard Euclidean cubes
        \(\sigma^d = (-1/2, 1/2)^d \) and
        \(\overline{\sigma}^d = [-1/2, 1/2]^d\)
        are \cubedec in \(\Rd\).
    \end{proposition}

    \begin{proof}
        Take the standard grid of half-open cubes of side
        \(\ell_k := 2^{-k}\) (with properly attached boundary
        for \(\sigma ^ d\) and \(\overline{\sigma} ^ d\)).
        Each such cube is connected, contains a 
        ball of radius \(\frac{1}{2}\ell_k\) and is contained in a 
        ball of radius \(\sqrt{d} \ell_k\), which gives
        properties (a)--(d) of \cref{def: connectedCubes}
        with scaling value \(1/2\).
    \end{proof}

    \subsection{Bi-Lipschitz deformations}

        The first strategy to obtain a \cubedec space
        is to deform another \cubedec one.

        \begin{definition} \label{def: biLip}
            We say that a map \(F\) between metric spaces
            \((X, \rho_X)\) and \((Y, \rho_Y)\) is 
            \textbf{bi-Lipschitz} with constant \(L\) if \(Y = F(X)\) 
            and 
            \[
                L^{-1} \rho_X(x_1, x_2)
                \le
                \rho_Y(F(x_1), F(x_2))
                \le
                L \rho_X(x_1, x_2).
            \]
            In this case, we also say that \((X, \rho_X)\) 
            and \((Y, \rho_Y)\) are bi-Lipschitz equivalent.
        \end{definition}

        \begin{lemma} \label{lem: biLipIn}
            Let \((X, \rho_X)\) and \((Y, \rho_Y)\) be metric spaces
            with bi-Lipschitz \(F \colon X \to Y\).
            Suppose that \(X' \subset X\) is \cubedec in \(X\).
            Then, \(F(X')\) is \cubedec in \(Y\).
        \end{lemma}

        \begin{proof}
            By \cref{cor: goodScaling}, \(X'\) admits a family
            \(\mathcal{P}\) of connected dyadic cubes with scaling value
            \(\delta\) and ball constant \(a\). For
            \(P_\alpha^k \in \mathcal{P}\), fix \(x_0 \in P_\alpha^k\) with
            \(B(x_0, a^{-1}\delta^k) \subset P_\alpha^k \subset
            B(x_0, a\delta^k)\), and set \(Q_\alpha^k := F(P_\alpha^k)\),
            \(y_0 := F(x_0)\). Since \(F\) is bi-Lipschitz with constant
            \(L\), we have 
            \(L^{-1}\rho_X(x_0, x) \le \rho_Y(F(x_0), F(x)) \le L\rho_X(x_0, x)\),
            which gives
            \[
                B(y_0, (La)^{-1}\delta^k) 
                \subset Q_\alpha^k 
                \subset
                B(y_0, La\delta^k),
            \]
            where the balls are taken in \(Y\).
            Thus, property (c) of \cref{def: connectedCubes} holds
            with scaling value
            \(\delta\) and ball constant \(La\). Since \(F\) is a
            homeomorphism, (a), (b), (d) are immediate, and
            \(F(X')\) is \cubedec in \(Y\).
        \end{proof}

        Taking \(X' = X\), we obtain a similar result
        for the whole space.

        \begin{corollary} \label{lem: biLip}
            Let \((X, \rho_X)\) and \((Y, \rho_Y)\) be metric spaces,
            and suppose that \(X\) is \cubedec.
            If there exists a bi-Lipschitz map
            \(F \colon X \to Y\), then \(Y\) is \cubedec too.
        \end{corollary}

        This tool allows us to construct sets that are
        \cubedec in \(\Rd\).

        \begin{corollary} \label{col: convex}
            A compact convex set \(A \subset \Rd\) with nonempty
            interior, equipped with the Euclidean metric,
            is \cubedec in \(\Rd\).
        \end{corollary}

        \begin{proof}
            Assume that \(0 \in \Int(A)\).
            Take the standard radial function
            \(R \colon \Sp^{d - 1} \to \R_+ \),
            \(R(\xi) := \sup \{ t \mid t\xi \in A \} \).
            A direct check shows that \(R\)
            is Lipschitz and satisfies
            \(0 < r_0 \le R \le r_1 < \infty\)
            for constants \(r_0, r_1\), such that
            \(B(0, r_0) \subset A \subset B(0, r_1)\).

            Consequently, the map
            \(F_A \colon \Rd \to \Rd \), defined by
            \(F_A(r \xi) := r \, R(\xi) \, \xi \), is
            bi-Lipschitz and satisfies \(F_A(\overline{\B}^d) = A\).
            Now consider \(F := F_A \circ F_{\overline{\sigma}^d} ^ {-1}\).
            Since \(F \colon \Rd \to \Rd\) is bi-Lipschitz
            and \(F(\overline{\sigma}^d) = A\),
            \cref{prop: cubeGood} and \cref{lem: biLipIn}
            give that \(A\) is \cubedec in \(\Rd\).
        \end{proof}

        The same argument works if \(A\) is an open convex
        set with compact closure.
        In particular, each closed nondegenerate
        simplex \(\overline{S} = \conv \{v_0, \dots, v_d\}\)
        and its interior are \cubedec in \(\Rd\).

    \subsection{Gluing}

        Next, we would like to glue \cubedec spaces together.

        \begin{lemma} \label{lem: glue}
            Suppose that \(X\) is a metric space, and 
            \(X = \bigcup_{j \in \Lambda} \overline{X}_j \), where
            the following conditions hold:
            \begin{itemize}
                \item[a)] Each \(\overline{X}_j\) with the restricted metric 
                is doubling;
                \item[b)]  the \(X_j\) are pairwise disjoint subsets that are
                \cubedec in \(X\) with the same scaling value;
                \(\delta\)
                \item[c)] Ball constants
                of the \(X_j\) are uniformly bounded.
            \end{itemize}
            Then, \(X\) is \cubedec.
        \end{lemma}

        \begin{proof}
            Construct sets \(Y_j\), where \(j \in \Lambda\).
            We attach each \(x \in X\) to some \(Y_j\).
            If \(x \in X_k\), take \(j = k\).
            Otherwise, take an arbitrary \(j\), such that
            \(x \in \overline{X}_j\).
            Since \(X_j \subset Y_j \subset \overline{X}_j\), \cref{prop: sandwich}
            gives that each \(Y_j\) is \cubedec in \(X\)
            with the same scaling
            value as \(X_j\) and bounded ball constant.
            By construction, \(X = \bigsqcup_{j \in \Lambda} Y_j\).

            \smallskip

            Let \(\mathcal{Q}_j\), \(j \in \Lambda\), be families
            of dyadic cubes for the \(Y_j\) with common scaling value \(\delta\)
            and ball constants \(C_j \le C_0\) from the hypothesis.
            Merge these families: put
            \(I_k := \bigsqcup_{j \in \Lambda} I_k ^ j\).

            This gives a family of connected dyadic cubes for \(X\).
            Property (a) of \cref{def: connectedCubes} holds since
            \[
                X = \bigsqcup_j Y_j
                =
                \bigsqcup_j \bigsqcup_{\alpha \in I_k ^ j} Q_\alpha ^{k, j}
                =
                \bigsqcup_{\alpha \in I_k} Q_\alpha ^ k.
            \]
            For (b), take \(\alpha \in I_l\), \(\beta \in I_k\), \(l \ge k\).
            If \(\alpha, \beta\) come from the same \(j\), this is
            property (b) for \(\mathcal{Q}_j\).
            Otherwise, suppose that they come from distinct \(j \neq j'\).
            Then, \(Q_\alpha ^ l \subset Y_j\) and
            \(Q_\beta ^ k \subset Y_{j'}\) are disjoint since \(Y_j\) and
            \(Y_{j'}\) are.
            Since each \(Y_j\) is \cubedec in \(X\), the inner
            inclusion of property (c) holds with the ball taken in
            \(X\), so (c) holds for the merged family with ball
            constant \(C_0\); (d) is trivial.
            Hence \(X\) is \cubedec.
        \end{proof}

        \begin{corollary} \label{col: finGlue}
            Suppose that \(X\) is a doubling metric space,
            \(X = \bigcup_{j = 1} ^ n \overline{X}_j \), where
            the \(X_j\) are pairwise disjoint subsets that are
            \cubedec in \(X\).
            Then, \(X\) is \cubedec.
        \end{corollary}

        \begin{proof}
            By \cref{cor: goodScaling} we may make scaling
            values the same. We note that the \(\overline{X}_j\) with the
            restricted metric are doubling.
            Then, since the union is finite,
            the ball constants are uniformly bounded,
            and we can apply \cref{lem: glue}.
        \end{proof}

    \subsection{Differential geometry}

        We conclude with several technical facts on \(C^1\)
        manifolds, needed later to handle the manifold case.

        Firstly, we need to equip manifolds with 
        some well-behaved metric.

        \begin{definitionbold} \label{def: lipComp}
            Suppose that \(M\) is a \(C^1\) manifold.
            We say that the metric \(\rho\) on \(M\)
            is \textbf{Lipschitz-compatible} if \(M\) may be covered
            by charts \((U, h)\) such that the map
            \[
                h \colon (U, \rho) \to (h(U), |\cdot|)
            \]
            is bi-Lipschitz.
            We call such a chart \textbf{bi-Lipschitz} as well.
        \end{definitionbold}

        \smallskip

        The two standard metrics on a manifold are both
        Lipschitz-compatible.

        \begin{proposition} \label{prop: geodesicLipComp}
            Suppose that \(M\) is a \(C^1\) manifold,
            equipped with a Riemannian metric \(g\).
            Then, the geodesic metric \(\rho_g\) on \(M\)
            is Lipschitz-compatible.
        \end{proposition}

        \begin{proof}
            Let \(m := \dim M\).
            Fix \(p \in M\), and an open \(U \ni p\), with
            \(h \colon U \to \HR^m \) or \(\R^m\).
            Assume that \(h(p) = 0\). Thus, 
            \(\overline{V} := \overline{B}(0, R) \) (or \(\overline{B}(0, R) \cap \HR^m \)
            in the boundary case) lies in \(h(U)\)
            for some \(R > 0\).
            Denote \(K := h^{-1}(\overline{V})\).

            \smallskip

            We write 
            \(\|v\| := |Dh_x(v)|\) for \(v \in T_x M\) (tangent space),
            and \(|v| := \sqrt{g_x(v, v)}\) for \(v \in T_x M\).
            \(K\) is compact, hence the set
            \(\{(x, v) \in T(M) \mid x \in K, \ \|v\| = 1\}\)
            is compact as well. Since \(|v|\) is nonzero, continuous
            and homogeneous, we obtain that
            \[
                L^{-1} \|v\|
                \le
                |v|
                \le
                L \|v\|
            \]
            for each \(x \in K, \, v \in T_x M\) and some \(L \ge 1\).

            \smallskip

            Take some \(x, y \in K\).
            We want to prove the upper and lower Lipschitz bounds
            on \(\rho_g(x, y)\).
            Let \(\beta \colon [0, 1] \to \overline{V}\) be a standard 
            segment between \(h(x)\) and \(h(y)\).
            Put \(\gamma := h^{-1} \circ \beta \). 
            We obtain that 
            \begin{equation} \label{eq: upperBound}
                \begin{split}
                    \rho_g(x, y)
                    \le
                    \len(\gamma)
                    & =
                    \int_0 ^ 1 |\gamma'(t)| \, dt \\
                    & \le
                    L \int_0 ^ 1 |\beta'(t)| \, dt
                    =
                    L |h(x) - h(y)|.
                \end{split}
            \end{equation}

            \smallskip 

            To get the lower bound, suppose that \(|h(x)|, |h(y)| < r\)
            for some fixed \(r < R/2\).
            Consider a path
            \(\gamma \colon [0, 1] \to M\) between \(x, y\).

            Suppose that \(\gamma\) leaves \(K\), so
            \(\gamma(t) \notin K\) for some \(t \in [0,1]\). 
            Take \(s := \inf \{t \mid \gamma(t) \notin K\}\).
            Then, \(z := \gamma(s) \notin \Int K\) and we obtain
            \begin{equation} \label{eq: leaveChartBound}
            \begin{split}
                \len(\gamma)
                & \ge
                \int_0^s |\gamma'(t)| \, dt \\
                & \ge
                L^{-1} |h(x) - h(z)| \\
                & \ge
                R/(2L) \ge r/L \ge (2L)^{-1} |h(x) - h(y)|.
            \end{split}
            \end{equation}

            If \(\gamma\) does not leave \(K\),
            we may apply the same estimate as in
            \eqref{eq: upperBound}, which gives
            \[
                \len(\gamma)
                \ge
                L^{-1} \len(\beta)
                \ge
                L^{-1} |h(x) - h(y)|
            \]
            for \(\beta := h \circ \gamma\). 
            
            In either case
            \(\len(\gamma) \ge (2L)^{-1} |h(x) - h(y)|\).
            Taking the infimum over all \(\gamma\), 
            we obtain the lower bound.

            \smallskip 
            
            We proved the bi-Lipschitz bound on 
            \(V_p := h^{-1}(B(0, r))\). Since \(p \in V_p\),  
            we found a bi-Lipschitz chart about \(p\),
            and the statement follows.
        \end{proof}

        \begin{proposition} \label{prop: embeddedLipComp}
            Suppose that a \(C^1\) manifold \(M\) is
            \(C^1\)-embedded in Euclidean space.
            Then, the restricted Euclidean metric
            \(\rho_M\) is Lipschitz-compatible.
        \end{proposition}

        \begin{proof}
            Again, let \(m := \dim M\), \(p \in M\),
            and let \(U \ni p\) be open, with
            \(h \colon U \to \HR^m \) or \(\R^m\),
            and \(h(p) = 0\).
            Suppose that \(f \colon M \to \R^D\) is
            an embedding.
            Then, the map \(f \circ h^{-1} \colon h(U) \to \R^D\) is
            a \(C^1\) map of rank \(m\), and therefore is bi-Lipschitz
            in some neighbourhood \(V_0 \subset \HR^m \) or \(\R^m\)
            of \(0\).
            Then, for each \(x, y \in U_p := h^{-1}(V_0)\), we have
            \[
                L^{-1} |h(x) - h(y)|
                \le 
                |f(x) - f(y)| 
                \le 
                L |h(x) - h(y)|
            \]
            with some \(L \ge 1\).
            Since \(|f(x) - f(y)| = \rho_M(x, y)\), the chart
            \(U_p\) is bi-Lipschitz, and the proposition follows.
        \end{proof}

        Next, we show that \(C^1\) diffeomorphisms between
        manifolds with Lipschitz-compatible metrics
        are bi-Lipschitz.

        \begin{lemma} \label{lem: biLipMap}
            Suppose that \(N\), \(M\) are \(C^1\) manifolds, equipped
            with Lipschitz-compatible (\cref{def: lipComp}) metrics 
            \(\rho_N\), \(\rho_M\).
            Let \(X \subset N\), \(Y \subset M\), and suppose that
            \(\overline{X}\) is a compact subset of \(N\). 
            Let \(F\) be a \(C^1\) diffeomorphism between some
            neighbourhoods of \(\overline{X}\) and \(\overline{Y}\),
            such that \(F(X) = Y\). Then
            \(F|_X \colon (X, \rho_N) \to (Y, \rho_M)\)
            is bi-Lipschitz.
        \end{lemma}

        \begin{proof}
            Since \(\overline{X}\) is compact,
            \(\overline{Y} = F(\overline{X})\) is compact as well.
            We prove the Lipschitz estimate for \(F\) and then
            apply the same argument for \(F^{-1}\).
            Consider a neighbourhood \(\Omega_X \supset \overline{X}\),
            in which \(F\) is a diffeomorphism.
            Put \(\Omega_Y := F(\Omega_X)\).

            \smallskip

            Consider some \(x \in \overline{X}\), and \(y := F(x)\).
            Choose compact coordinate neighbourhoods
            \(U_x \subset \Omega_X\) of \(x\) and
            \(U_y \subset \Omega_Y\) of \(y\).
            Fix \(C^1\) bi-Lipschitz charts
            \(h_x \colon U_x \to V_x\), \(h_y \colon U_y \to V_y\).
            Shrinking \(U_x\), \(U_y\) if necessary, we assume
            that \(V_x, V_y\)
            are closed (half) balls, and that
            \(F(U_x) \subset U_y\).
            Then, the map 
            \(G := h_y \circ F \circ h_x ^ {-1} \colon V_x \to V_y\) is
            of class \(C^1\).
            The differential norm bound yields that \(G\) is Lipschitz.
            Indeed, for each \(p, q \in V_x\), the mean value
            inequality (applicable since \(V_x\) is convex) gives
            \[
                |G(p) - G(q)|
                \le 
                \sup_{t \in V_x} \|D G_t \| |p - q| 
                \le 
                L_x |p - q|.
            \]
            Thus, \(F\) is Lipschitz on \(U_x\) with some 
            constant \(L_x '\), which may depend on \(x\).

            \smallskip

            Suppose now that
            \[
                \frac{\rho_M(F(x_n), F(y_n))}{\rho_N(x_n, y_n)}
                \to \infty
            \]
            for some \(x_n, y_n \in X\).
            Passing to a subsequence, we obtain
            \(x_n \to x^*\) and \(y_n \to y^*\) for some
            \(x^*, y^* \in \overline{X}\). Since \(\rho_M(F(x_n), F(y_n))\) is bounded
            above (\(\overline{Y}\) is compact), it follows that
            \(\rho_N(x_n, y_n) \to 0\) and \(x^* = y^* =: p\).
            We get \(x_n, y_n \in U_p\)
            for large enough \(n\).
            Consequently, the ratio is at most \(L_p '\), which
            is a contradiction, so \(F\) is Lipschitz, and
            \(F^{-1}\) is Lipschitz analogously.
        \end{proof}

%% file: chapters/manifolds.tex
\section{Manifold case} \label{sec: manifolds}

    Now we are ready to apply the preceding machinery to
    obtain concrete families of \cubedec spaces.
    We will focus on manifolds, both
    abstract and embedded in Euclidean space \(\Rd\).

    \medskip

    \subsection{Bi-Lipschitz graphs}

        The most naive approach is to consider a
        metric space that is a
        bi-Lipschitz image of
        the whole Euclidean space \(\Rd\).
        It immediately follows from \cref{lem: biLip}
        that any space of this type is \cubedec.
        Examples include
        graphs of Lipschitz functions.
        Such graphs are the basic building
        blocks of the theory of uniform rectifiability
        developed in \cite{DS97, DS91}.

        \begin{proposition} \label{prop: lipGraph}
            Let \(f \colon \R^{d-1} \to \R\) be Lipschitz with
            constant \(L\). Then the graph
            \[
                \Gamma_f := \{ (x, f(x)) : x \in \R^{d-1} \} \subset \Rd,
            \]
            equipped with the Euclidean metric,
            is bi-Lipschitz equivalent to \(\R^{d-1}\).
        \end{proposition}

        \begin{proof}
            Consider \(\Phi \colon \R^{d-1} \to \Gamma_f\) given by
            \(\Phi(x) = (x, f(x))\).
            For \(x, y \in \R^{d-1}\),
            \[
                |x - y|
                \le
                |\Phi(x) - \Phi(y)|
                =
                \sqrt{|x-y|^2 + |f(x) - f(y)|^2}
                \le
                \sqrt{1 + L^2} \, |x - y|,
            \]
            so \(\Phi\) is a bi-Lipschitz bijection onto \(\Gamma_f\).
        \end{proof}

        More exotic examples include
        certain Alexandrov spaces
        with additional properties, studied
        in \cite{BL03}.

        \medskip

        The bi-Lipschitz-image approach, however, does
        not exhaust all manifolds of interest.
        Even among metric spaces that are otherwise
        well-behaved, there exist spaces
        admitting no bi-Lipschitz parametrisation by
        \(\Rd\), as shown in \cite{Sem96} (see also \cite{Sem93}).
        This motivates the triangulation-based
        argument developed below, which may be applied to
        manifolds with nontrivial topology.

        \medskip

    \subsection{Simplicial complexes}

        To deal with manifolds, we need to split
        them into simplices.
        Recall the definition of a triangulation.

        \begin{definition} \label{def: simplComplex} [\cite{Mun66}, Def. 7.1].
            Recall that a \textbf{(simplicial) complex} 
            \(K\) is a collection of
            (open) nondegenerate simplices in \(\R^D\) such that
            \begin{itemize}
                \item[(1)] every face of a simplex of \(K\) is in \(K\);
                \item[(2)] the intersection of two closed simplices 
                    of \(K\) is either empty or a face of each of them;
                \item[(3)] each point of \(|K| := \bigcup_{S \in K}
                    S\) has a neighbourhood meeting only finitely
                    many simplices of \(K\).
            \end{itemize}
            We call \(K\) \textbf{finite} if it
            has finitely many simplices.
            We call a simplex \(S \in K\) \textbf{maximal}
            if it is not a face of another simplex from \(K\).
        \end{definition}

        We equip a connected simplicial complex with the
        following metric \(\rho_K \).
        On each simplex \(S \in K\)
        of dimension \(m\) we set
        \[
            \rho_S(F(x), F(y))
            :=
            |x - y|
        \]
        where \(F \colon \overline{\Delta}^m \to \overline{S} \)
        is the affine parametrisation of \(S\).
        It is easy to check that these definitions agree
        on the shared faces of adjacent simplices.
        We extend \(\rho_S\) to \(|K|\) as an
        intrinsic metric. Formally, put 
        \[
            \rho_K(x, y) 
            := 
            \inf \left( \sum_{j = 0}^{n - 1} 
            \rho_{S_j} (x_j, x_{j + 1}) \right),
        \]
        where the infimum is taken over all \(x_0, \dots, x_n \in |K|\)
        such that \(x_0 = x, \ x_n = y\) and
        \(x_j, x_{j + 1}\) lie in a common closed simplex
        \(\overline{S}_j\).
        One may check that \(\rho_K\) and \(\rho_S\) 
        agree on \(\overline{S}\) for each \(S \in K\).

        Next, we give definitions of types of
        triangulations, which we are working with.

        \begin{definition} \label{def: triangulation}
            [\cite{Mun66}, Def. 8.1, 8.3, Thm. 8.4].
            Let \(K\) be a simplicial complex and \(M\) a
            \(C^1\) manifold. 
            A map \(f \colon |K| \to M\) is of
            \textbf{class \(C^1\)}
            if, for each simplex \(S \in K\),
            \(f|_{\overline{S}} \)
            is of class \(C^1\),
            and \textbf{nondegenerate}
            if, in addition, \(f|_{\overline{S}}\)
            has rank \(\dim S\) at every point of \(\overline{S}\).

            A \textbf{\(C^1\) triangulation} of \(M\) is
            a homeomorphism \(\tau \colon |K| \to M\)
            that is nondegenerate and of class \(C^1\).
            It is \textbf{finite} if \(K\) is finite.

            \smallskip

            Analogously, if \((X, \rho)\) is a metric space, we
            call \(\tau \colon (|K|, \rho_K) \to (X, \rho) \) 
            a \textbf{bi-Lipschitz triangulation} 
            of \((X,\rho)\)
            if, instead, \(\tau\) is bi-Lipschitz.
        \end{definition}

        We will use the following well-known lemma.

        \begin{lemma} \label{lem: complexCompact}
            \(|K|\) is compact if and only if
            \(K\) is finite.
        \end{lemma}

        \begin{proof}
            If \(K\) is finite, \(|K|\) is a finite union of
            closed simplices \(\overline{S}\), each compact,
            so \(|K|\) is compact.
            Conversely, the neighbourhoods from property (3)
            of \cref{def: simplComplex}
            form an open cover of \(|K|\).
            It does not have a finite subcover when \(K\) is infinite,
            since each neighbourhood meets only finitely 
            many simplices.
        \end{proof}

        \medskip

        The next lemma is the main tool to deal with
        manifolds.
        
        \begin{lemma} \label{lem: complexCubeDecomp}
            Let \(K\) be a connected simplicial complex.
            Then, \((|K|, \rho_K)\) is \cubedec.
        \end{lemma}

        \begin{proof}
            Consider all maximal simplices of \(K\),
            denoted by \(\{ S_i \}_{i \in \Lambda} \).
            Let \(\theta \in (0, 1)\).
            For each \(m\), fix a connected
            dyadic cube decomposition of \(\Delta^m \) in \(\R^m\)
            (it is \cubedec in \(\R^m\) by \cref{col: convex})
            with scaling value \(\theta\), which is possible
            by \cref{cor: goodScaling}.

            \smallskip

            Fix some \(S = S_i\).
            Denote by \(A \cong \R ^ m\) the affine hull of \(S\).
            Consider the affine parametrisation
            \(F \colon \R^m \to A\), with \(F(\Delta^m) = S\).
            We denote by \(\rho_A\) the metric on \(A\)
            induced by \(F\). On \(S\) it agrees with \(\rho_S\),
            considered in the definition of \(\rho_K\).
            Since \(F\) is an isometry, \cref{lem: biLipIn} yields that 
            \(S\) admits a dyadic cube decomposition \(\mathcal{Q}\) with
            scaling value \(\theta\) and some ball constant
            \(C_m > 0\), depending only on dimension \(m\).
            Moreover, \(\mathcal{Q}\) is a dyadic cube decomposition
            of \(S\) in \(A\).

            \smallskip

            We claim that it is also a decomposition in \(|K|\).
            Indeed, consider some dyadic cube \(Q \in \mathcal{Q}\) 
            with inner ball \(B = B(x, r)\), where \(x \in Q\).
            We prove that \(B\), considered in \(|K|\), 
            is a ball with centre \(x\) and radius \(r\) as well.
            Take some \(y \in |K|\) with \(\rho_K(x, y) < r\).
            Suppose that \(y \notin \overline{S}\).
            Consider a chain \(x = x_0, \dots, x_n = y\)
            with \(x_j, x_{j + 1} \in \overline{S}_j\) and
            \(s := \sum_{j = 0}^{n - 1} \rho_{S_j}(x_j, x_{j + 1}) < r \).
            Take the minimal \(l \ge 0\) with \(x_{l + 1} \notin \overline{S}\)
            (it exists because \(x_n = y \notin \overline{S} \) ).
            We have that \(x_l \in \overline{S}\) and \(x_{l + 1} \notin \overline{S}\), 
            hence \(S_l \neq S\) and \(x_l \in \overline{S} \cap \overline{S}_l\),
            which is some (proper) face of the maximal simplex \(S\).
            Therefore, \(x_l \in \partial S\). 
            Since \(x_0, \dots, x_l \in \overline{S}\) and \(\rho_K\)
            agrees with \(\rho_A \) on \(\overline{S}\), the triangle inequality yields
            \[ 
                \rho_A(x, x_l)
                =
                \rho_K(x, x_l)
                \le 
                \sum_{j = 0}^{l - 1} \rho_{S_j}(x_j, x_{j + 1})
                \le 
                s < r,
            \]
            which is impossible since \(B \subset Q \subset S\).
            We conclude that \(y \in \overline{S}\) and
            \(\rho_A(x, y) = \rho_K(x, y) < r\),
            so \(y \in B\).
            Consequently, \(S\) is \cubedec in \(|K|\).

            \smallskip

            By the definition of a complex,
            \(K\) is embedded in Euclidean space \(\R^D\),
            hence \(m \le D\).
            We have that 
            \(|K| = \bigcup_{i \in \Lambda} \overline{S}_i \),
            the \(S_i\) are pairwise disjoint, each
            \(\overline{S}_i \) with \(\rho_K\) is isometric
            to \(\overline{\Delta}^m\), hence compact and doubling.
            Consequently, \cref{lem: glue} yields that \(|K|\) 
            is \cubedec.
        \end{proof}

        \medskip

    \subsection{Triangulations}

        The important lemma, which we would like to
        apply to manifolds, is stated as follows.

        \begin{lemma} \label{lem: biLipTriang}
            Suppose that \(X\) is a metric space
            that admits a bi-Lipschitz
            triangulation. Then, \(X\) is \cubedec.
        \end{lemma}

        \begin{proof}
            Suppose that
            \(\tau \colon (|K|, \rho_K) \to (X, \rho)\)
            is bi-Lipschitz.
            By \cref{lem: complexCubeDecomp}, \(|K|\) with the metric
            \(\rho_K\) is \cubedec, so, by \cref{lem: biLip},
            \(X\) is \cubedec as well.
        \end{proof}

        \begin{definition} \label{def: boundedGeometry} [\cite{CGT82, Eic07}].
            A smooth Riemannian manifold \(M\) without boundary 
            has \textbf{bounded geometry}
            if its sectional curvature is uniformly bounded,
            and its injectivity radius is bounded away from zero.
        \end{definition}

        It follows from the result [\cite{Att94}, Thm. 1.14] 
        that smooth manifolds
        (not necessarily compact) with bounded geometry
        admit a bi-Lipschitz triangulation, and therefore
        are \cubedec by \cref{lem: biLipTriang}.
        A more recent approach, based on Delaunay
        triangulations, together with precise constants
        and further generalisations, is developed
        in \cite{BDG18, Bow20}.

        \medskip

        A similar fact for compact spaces
        holds without any additional geometric
        assumption.

        \begin{theorem} \label{thm: geodesicCubeDecomp}
            Suppose that \(M\) is a compact connected \(C^1\) manifold,
            equipped with a Riemannian metric \(g\).
            Then, \(M\) with the geodesic metric \(\rho_g\)
            is \cubedec.
        \end{theorem}

        \begin{proof}
            By the classical Cairns--Whitehead triangulation theorem
            \cite{Cai35, Whi40} (a more modern version is presented
            in [\cite{Mun66}, Thm. 10.6]), a \(C^1\) manifold
            \(M\) admits a \(C^1\) triangulation
            \(\tau \colon |K| \to M\).
            Since \(M\) is compact, \(|K| = \tau^{-1}(M)\)
            is compact, and therefore \(K\) is finite
            by \cref{lem: complexCompact}.

            \smallskip

            Let \(S \in K\) be a maximal simplex, 
            so \(\dim S = \dim M =: d\). 
            Denote \(Y := \tau(S)\).
            Let \(A \cong \R^d\) be the affine hull of
            \(S\) with the Euclidean metric.

            A standard argument [\cite{Mun66}, Thm. 1.5]
            gives that \(\tau|_{\overline{S}}\)
            extends to a \(C^1\) map \(h\)
            from some neighbourhood \(\Omega_S \supset \overline{S}\) 
            in \(A\) to \(M\).
            Since \(\tau\) is nondegenerate, \(h\) has rank \(d\)
            at every point of \(\overline{S}\), and, since the set
            \(\{x \mid \det D h (x) \neq 0 \}\) is open, we may shrink
            \(\Omega_S\) so that \(h\) has rank \(d\) on all of \(\Omega_S\).
            Finally, we check that \(h\) is injective
            on some neighbourhood of \(\overline{S}\).
            Suppose not. Then \(h(x_n) = h(y_n)\) for some sequences
            \(x_n, y_n \in A\) with \(x_n \neq y_n\), such that
            \(\dist(x_n, \overline{S}) \to 0\) and \(\dist(y_n, \overline{S}) \to 0\).
            Passing to a subsequence, we obtain
            \(x_n \to x^* \in \overline{S}\), \(y_n \to y^* \in \overline{S}\).
            By continuity, \(\tau(x^*) = h(x^*) = h(y^*) = \tau(y^*)\),
            so, since \(\tau\) is injective, \(x^* = y^* =: p\).
            This contradicts the fact that \(h\) is a diffeomorphism on a
            neighbourhood of \(p\).
            Hence, \(h \colon \Omega_S \to M\) is an embedding, i.e. a
            diffeomorphism from \(\Omega_S\) onto its image.

            We may therefore apply \cref{lem: biLipMap} to
            \(A\) (with the Lipschitz-compatible Euclidean metric),
            and to \(M\) with \(\rho_g\), which is Lipschitz-compatible
            by \cref{prop: geodesicLipComp}.
            This gives that \(\tau|_{\overline{S}}\) is a bi-Lipschitz map from
            \((\overline{S}, |\cdot|) \) to \((\overline{Y}, \rho_g)\). Since the identity map 
            \(\mathrm{Id} \colon (\overline{S}, \rho_K) \to (\overline{S}, |\cdot|) \) 
            is bi-Lipschitz, the composition is bi-Lipschitz
            as well with some constant \(L(S)\).
            Put \(L := \max_{S \in K, \, \dim S = d} L(S)\).
            Since \(K\) is finite, \(L < +\infty\).
            Note that the same constant works for nonmaximal
            simplices, since each of them is a proper face of a
            maximal one with restricted metric. Hence,
            \[ 
                L^{-1} \rho_K(p, q) 
                \le 
                \rho_g(\tau(p), \tau(q)) 
                \le 
                L \rho_K(p, q)
            \]
            for each simplex \(S \in K\) and \(p, q \in \overline{S}\).

            \smallskip

            We claim that
            \(\tau \colon (|K|, \rho_K) \to (M, \rho_g)\) is
            globally \(L\)-bi-Lipschitz.
            Take some \(x, x' \in |K|\), and let 
            \(y = \tau(x), \ y' = \tau(x')\).
            Firstly, we prove that \(\rho_g(y, y') \le L \rho_K(x, x')\).
            Consider any \(x_0, \dots, x_n \in |K|\) such that
            \(x_0 = x\), \(x_n = x'\), and \(x_j, \ x_{j + 1}\) lie in a
            common closed simplex. 
            We obtain that 
            \[
                \sum_{j = 0}^{n - 1} \rho_K(x_j, x_{j + 1})
                \ge
                L^{-1} \sum_{j = 0}^{n - 1} \rho_g(\tau(x_j), \tau(x_{j + 1}))
                \ge 
                L^{-1} \rho_g(y, y').
            \]
            Taking the infimum over such sequences \(x_j\) gives the 
            upper bound.

            \smallskip

            Next, we prove the reverse inequality, written as
            \(L^{-1} \rho_K(x, x') \le \rho_g(y, y')\).
            Consider a path \(\gamma \colon [0, 1] \to M\)
            from \(y\) to \(y'\), and denote
            \(\beta := \tau^{-1} \circ \gamma \colon [0, 1] \to |K|\).
            We construct a sequence \(t_j \in [0, 1]\).
            We write \(x_j := \beta(t_j) \in |K|\) and 
            \(y_j := \gamma(t_j) \in M\).

            Put \(t_0 := 0\). Given \(t_j < 1\), define \(t_{j + 1}\) as
            follows.
            Consider the open star of \(x_j\),
            denoted by
            \[
                T := \bigcup \{S \colon \ x_j \in \overline{S}\}.
            \]
            Since \(T\) is open, for \(t\) 
            close enough to \(t_j\) we have 
            \(\beta(t) \in T\).
            Fix some \(t^* > t_j\) with \(x^* := \beta(t^*) \in T\).
            Take a simplex \(S_j\) such that
            \(x_j, \ x^* \in \overline{S}_j\).
            Put \(t_{j + 1} := \sup \{ t \in [0, 1]
            \mid \beta(t) \in \overline{S}_j \}\).
            By the choice of \(S_j\), we have \(t_j < t^* \le t_{j + 1}\).
            Consequently, no simplex can be picked twice
            (otherwise, \(t_i < t_j \) and \(\beta(t_j) \in \overline{S}_i \),
            which contradicts the choice of \(t_{i + 1}\)).
            Since there are finitely many simplices,
            we make only finitely many steps.

            The process finishes only when \(t_n = 1\).
            Hence, \(x_0 = x\) and \(x_n = x'\).
            By definition, \(x_j\) and \(x_{j + 1}\) lie 
            in \(\overline{S}_j\), hence
            \[
                \begin{split}
                    \len(\gamma)
                    & = 
                    \sum_{j = 0}^{n - 1} \len(\gamma|_{[t_j, t_{j + 1}] }) \\
                    & \ge  
                    \sum_{j = 0}^{n - 1} \rho_g(y_j, y_{j + 1}) \\
                    & \ge 
                    \sum_{j = 0}^{n - 1} L^{-1} \rho_K(x_j, x_{j + 1}) \\
                    & \ge 
                    L^{-1} \rho_K(x, x'). 
                \end{split}
            \]
            Taking the infimum over \(\gamma\), we obtain the
            lower bound on \(\rho_g(y, y')\).

            By \cref{lem: biLipTriang}, \(M\) is therefore
            \cubedec.
        \end{proof}

        Note that an Ahlfors--David regular positive
        measure on a compact metric space is finite.
        Therefore, \cref{thm: geodesicCubeDecomp}, combined with
        \cref{thm: partitionConnectedIntro},
        proves \cref{thm: manifoldIntro}.

        Next, we extend the preceding result to the
        case of an arbitrary Lipschitz-compatible metric.

        \begin{corollary} \label{cor: lipCompCubeDecomp}
            Suppose that \(M\) is a compact connected \(C^1\) manifold
            with a Lipschitz-compatible metric \(\rho\).
            Then, \((M, \rho)\) is \cubedec.
        \end{corollary}

        \begin{proof}
            Fix a Riemannian metric \(g\) on \(M\).
            \Cref{thm: geodesicCubeDecomp} gives that \((M, \rho_g)\)
            is \cubedec.
            Consider the identity map \(\mathrm{Id}\) between
            \((M, \rho_g)\) and \((M, \rho)\).
            Since it is a \(C^1\) diffeomorphism, and
            \(\rho_g\) and \(\rho\) are Lipschitz-compatible,
            \cref{lem: biLipMap} gives that \(\mathrm{Id}\) is bi-Lipschitz.
            Hence, by \cref{lem: biLip}, \((M, \rho)\) is \cubedec.
        \end{proof}

        Since the restricted metric on an embedded manifold
        is Lipschitz-compatible (see \cref{prop: embeddedLipComp}),
        we immediately obtain the result in this case.

        \begin{corollary} \label{cor: embeddedCubeDecomp}
            Suppose that \(M\) is a compact connected \(C^1\)
            manifold \(C^1\)-embedded in Euclidean space.
            Then, \(M\) with the restricted Euclidean metric 
            is \cubedec.
        \end{corollary}

%% file: chapters/questions.tex
\section{Open questions and final remarks} \label{sec: questions}
    \subsection*{Weighted decomposition}
        We do not know whether a weighted decomposition
        (\cref{def: weightedDecomp})
        is possible for all Ahlfors--David regular
        metric measure spaces, as is the case for equal measures.
        Although our algorithm
        crucially depends on the additional property that
        dyadic cubes are connected, we conjecture that
        weighted decomposition is possible in the most general
        case.

    \medskip

    \subsection*{\Cubedec spaces}
        Another direction of research is to study
        \cubedec spaces in their own right.
        Even for manifolds, the simple bi-Lipschitz approach
        we developed is restricted to the compact case, or
        to the case of sufficient smoothness.
        A typical obstruction for a metric space
        to be \cubedec is a ``very nongeodesic''
        metric.

        However, there exist spaces (say, a circle without
        a point, embedded in the plane)
        that are not geodesic but are \cubedec.
        Laakso spaces \cite{Laa00} are doubling, geodesic, and not
        bi-Lipschitz equivalent to any subset of a Euclidean space,
        so our techniques cannot be applied to them.
        However, one may verify that they are \cubedec.
        We do not know whether doubling geodesic 
        metric spaces are \cubedec in general.
        We would also like to find more natural and complete
        necessary or sufficient criteria for a space (for example,
        a domain in \(\Rd\) with the Euclidean metric)
        to be \cubedec.

    \medskip

    \subsection*{Applications}
        Like equal-measure partitions, weighted decompositions
        may have applications in both pure and applied science.
        We came to this topic while studying certain 
        problems in calculus,
        and we hope that the results developed here will be used in
        a forthcoming paper.

    \medskip

    \subsection*{Acknowledgements}
        The author is grateful to his advisor Ioann Vasilyev, 
        who suggested this topic, reviewed this paper, and found some
        mistakes and inaccuracies.
        Without him this work would never even have 
        started.

    \subsection*{AI use in the paper}
        Large language models (Claude Sonnet 5 and Opus 5)
        were used, in part, to polish the language
        of this paper, find references, and check the drafted proofs.
        The general strategy of the research, and
        all mathematical statements and proofs,
        were developed and verified by the author.